\documentclass[12pt]{amsart}
\usepackage{amsmath,amssymb,amsthm,amscd}

\newtheorem{theorem}{Theorem}[section]
\newtheorem{proposition}[theorem]{Proposition}
\newtheorem{lemma}[theorem]{Lemma}
\newtheorem{corollary}[theorem]{Corollary}
\newtheorem{remark}[theorem]{Remark}
\newtheorem{definition}[theorem]{Definition}
\newtheorem{example}[theorem]{Example}

\newcommand{\C}{{\mathbb C}}

\newcommand{\Q}{{\mathbb Q}}
\newcommand{\R}{{\mathbb R}}
\newcommand{\Z}{{\mathbb Z}}

\renewcommand{\L}{{\mathcal L}}

\newcommand{\Pl}{{\mathbb P}}
\newcommand{\Qbar}{\bar{\Q}}

\DeclareMathOperator{\Spec}{Spec}

 \DeclareMathOperator{\Pic}{Pic}

 \DeclareMathOperator{\N}{N}
\DeclareMathOperator{\Pico}{Pic}

\newcommand{\vfi}{{\varphi}}

\newcommand{\bQ}{{\mathbb Q}}

\newcommand{\cO}{\mathcal{O}}

\newfont{\cyrr}{wncyr10}
\title{Canonical metrics of commuting maps}
\author{J.~Pineiro}
\thanks{I would like to express my gratitude to Professor Lucien Szpiro for the valuable discussions on the subject of this paper.
\\ The author was partially supported by the PSC-CUNY Award 60029-3637}


\address{
Department of Mathematics and Computer Science\\
Bronx Community College of CUNY\\
University Ave. and West 181 Street\\
Bronx, NY 10453}
 \email{jorge.pineiro@bcc.cuny.edu}

\begin{document}

\begin{abstract}
Let $\vfi : X \rightarrow X$ be a map on an projective variety. It
is known that whenever $\vfi^*: \Pico(X) \otimes \R \rightarrow
\Pico(X) \otimes \R$ has an eigenvalue $\alpha
> 1$, we can build a canonical measure, a canonical height and a canonical
metric associated to $\vfi$.  In the present work, we establish the
following fact: if two \textbf{commuting maps} $\vfi, \psi : X
\rightarrow X$ satisfy these conditions, for eigenvalues $\alpha$
and $ \beta$ and the same eigenvector $\L$, then the canonical
metric, the canonical measure, and the canonical height associated
to both maps, are identical.
\end{abstract}

\maketitle

 \section{Introduction}
Let $X$ be a projective variety defined over a number field $K$.
 Suppose that $\vfi : X \rightarrow X$ is a map on $X$, also defined over $K$.  Assume that we can find an ample line bundle
$\L$ on $X$ and a number $\alpha > 1$, such that $ \L^{\alpha} \cong
\varphi^*\L$. Under this conditions, we can build the canonical
height $\hat{h}_{\vfi}$ (\cite{silvermancanonicalheights} theorem
1.1)
 and the canonical measure $d\mu_{\vfi}$
(\cite{dynamics} proposition 3.1.4) associated to $\vfi$ and $\L$.
They satisfy nice properties with respect to the map $\vfi$, for
example we have $\hat{h}_{\vfi} \circ \vfi=\deg(\vfi)
\hat{h}_{\vfi}$ and $\vfi_* \mu_{\vfi}=\mu_{\vfi}.$ Sometimes it
happens that a whole set of maps are associated to the same
canonical height function and measure.  As our first example
consider the collection of maps $\phi_k : \Pl_{\bar{K}}^1
\rightarrow \Pl_{\bar{K}}^1$ on the Riemann Sphere, where $\phi_k$
is defined as $\phi_k(t)=t^k$. The line bundle $\L=\cO(1)$ on
$\Pl^1$ satisfies the isomorphism $\phi_k^* \L \cong \L^k$. If one
builds the canonical height and measure associated to $\phi_k$ and
$\cO(1)$, one obtains:
\begin{enumerate}
\item All $\phi_k$ have the same canonical height namely, the naive height
$h_{nv}$ on $\Pl^1_{\bar{\Q}}$.  The naive height $h_{nv}(P)$ is a
refined idea of the function $\sup \{|a_0|,|a_1| \}$, measuring the
computational complexity of the projective point $P=(a_0:a_1)$. For
a precise definition see later definition \ref{naive height}.
\item All $\phi_k$ have the same canonical measure, that is, the Haar measure $d\theta$
on the unit circle $S^1$.
\end{enumerate}
Similar properties are fulfilled by the collection of maps $[n] : E
\rightarrow E$, representing multiplication by $n$ on an elliptic
curve $E$ defined over $K$.  If $\L$ is an ample symmetric line
bundle on $E$, we have the isomorphism $[n]^* \L \cong \L^{n^2}$,
along with the properties:
\begin{enumerate}
\item All maps $[n]$ share the same canonical height, that is, the Neron-Tate height
$\hat{h}_{E}$ on $E$.  In fact this will be our definition
(\ref{NT}) of the Neron-Tate height on $E$.  For many other
interesting properties we refer to B-4 in \cite{intro-diophantine}.
\item All maps $[n]$ have the same canonical measure, that is, the Haar measure $i/2 Im(\tau) dz \wedge d\bar{z}$ on
$E = \C / \Z+\tau \Z$.
\end{enumerate}
We observe that any two maps in each collection commute for the
composition of maps.  Besides, the line bundle $\L \in \Pic(X)
\otimes \R$, suitable to make everything work, is the same within
each collection.  The present work establish the general fact:
\begin{proposition} Let $X$ be a projective variety defined over a number field $K$.  Suppose that two
maps $\varphi, \psi : X \rightarrow X$  commute ( $\varphi \circ
\psi = \psi \circ \varphi$) and satisfy the following property: For
some ample line bundle $\L \in \Pic(X) \otimes \R$ and real numbers
$\alpha, \beta
> 1$, we have
  $\varphi^*\L \xrightarrow{\sim} \L^{\alpha}$ and $\psi^*\L \xrightarrow{\sim} \L^{\beta}$, then we have
   $ \hat{h}_{\vfi}=\hat{h}_{\psi}=\hat{h}_{\varphi \circ
\psi}$ and
   $ d\mu_{\vfi}=d\mu_{\psi}=d\mu_{\vfi \circ \psi}$.
\end{proposition}
This result is known in dimension one, a proof can found for example
in \cite{eremenko}.  Also it is a well known fact
\cite{intro-diophantine}, that commuting maps in a projective
variety must share the same canonical height. The main feature of
the present work it is to obtain all this results from the equality
of the canonical metrics.  Given a ample line bundle $\L$ on $X$, it
was an original idea of Arakelov \cite{arakelovinter} to put metrics
on $\L_{\sigma}=\L \otimes_{\sigma} \C$ over all places $\sigma$ of
$K$ at infinity. This gave rise to heights as intersection numbers
and curvature forms at infinity.  In was then an idea of Zhang
\cite{zhangadelic} to look for suitable metrics at all places $v$ of
$K$.  In presence of the dynamics $\vfi : X \rightarrow X$, the line
bundle $\L$ on $X$ can be endowed \cite{zhangadelic} with very
special metrics $\|.\|_{\vfi,v}$ on $\L_v$ that satisfy the
functional equation
$$\|.\|_{\varphi,v}=(\phi^*\varphi^*\|.\|_{\varphi,v})^{1/\alpha},$$
whenever we have an isomorphism $\phi : \L^{\alpha}
\xrightarrow{\sim} \varphi^*\L$.  The canonical height and the
canonical metric will be defined (definitions \ref{canonical
measure} and \ref{canonical height as intersection}) depending only
on the metric $\|.\|_{\vfi}$. The equality of canonical heights and
measure for commuting maps is a consequence of the following
proposition:
\begin{proposition}
Suppose that two maps $\varphi, \psi : X \rightarrow X$ commute, and
for some ample line bundle $\L \in \Pic(X) \otimes \R$ we have
$\varphi^*\L \xrightarrow{\sim} \L^{\alpha}$ and $\psi^*\L
\xrightarrow{\sim} \L^{\beta}$ for some numbers $\alpha, \beta > 1$,
then $\|.\|_{\vfi}=\|.\|_{\psi}$.
\end{proposition}
Towards the end of the paper we discuss maps on $\Pl^1$ arising as
projections of maps on elliptic curves with complex multiplication.
 We study ramification points and present examples of commuting maps on the Riemann sphere.
 \section{Canonical heights and canonical measures}
\subsection{Canonical metrics}
Consider the projective variety $X$ defined over a number field $K$,
a map $\vfi : X \rightarrow X$ defined over $K$, and an ample line
bundle $\L \in \Pic(X) \otimes \R$ such that $\phi : \L^{\alpha}
\xrightarrow{\sim} \varphi^*\L$ for some $\alpha
>1$.  This situation will be called \cite{dynamics} a polarized dynamical system $(X,\vfi,\L,\alpha)$ on $X$ defined over $K$. \\ Assume that for every place $v$ of $K$ we have chosen a continuous and bounded metric
$\|.\|_v$ on each fibre of $\L_v=\L \otimes_K K_v$. The following
theorem is proposition 2.2 in \cite{zhangadelic}:
\begin{theorem} \label{canonical metric}
The sequence defined recurrently by $\|.\|_{v,1}=\|.\|_v$ and
$\|.\|_{v,n}=(\phi^* \varphi^* \|.\|_{v,n-1})^{1/\alpha}$ for $n
> 1$, converge uniformly on $X(\bar{K}_v)$ to a metric
$\|.\|_{v,\varphi}$ (independent of the choice of $\|.\|_{v,1}$) on
$\L_v$ which satisfies the equation
$\|.\|_{\varphi,v}=(\phi^*\varphi^*\|.\|_{\varphi,v})^{1/\alpha}$.
\end{theorem}
\begin{proof}
 Denote by $h$ the continuous function $ \log \frac{\|.\|_2}{\|.\|_1}$ on $X(\bar{K}_v)$.
Then
$$\log\|.\|_n=\log\|.\|_1 + \sum_{k=0}^{n-2}(\frac{1}{\alpha} \phi^*\varphi^*)^k h.$$
Since $\|(\frac{1}{\alpha} \phi^*\varphi^*)^k h\|_{sup} \leq
(\frac{1}{\alpha})^k\|h\|_{sup}$, it follows that the series given
by the expression $\sum_{k=0}^{\infty} ( \frac{1}{\alpha}
\phi^*\varphi^*)^k h$, converges absolutely to a bounded and
continuous function $h^v$ on $X(\bar{K}_v)$. Let
$\|.\|_{\varphi,v}=\|.\|_1 \exp (h^v) $, then $\|.\|_n$ converges
uniformly to $\|.\|_{\varphi,v}$ and its not hard to check that
$\|.\|_{\varphi,v}$ satisfies
$$\|.\|_{\varphi,v}=(\phi^*\varphi^*\|.\|_{\varphi,v})^{1/\alpha},$$
which was the result we wanted to prove.
\end{proof}
\begin{definition}
The metric $\|.\|_{v,\varphi}$ is called the canonical metric on
$\L_v$ relative to the map $\vfi$.
\end{definition}

\begin{example}
Consider the line bundle $\L=\cO_{\Pl^n}(1)$ on $\Pl_{\Qbar}^n$ and
the rational map $\phi_k : \Pl_{\Q}^n \rightarrow \Pl_{\Q}^n$ given
by the expression $\phi(T_0:...:T_n)=(T_0^k:...:T_n^k)$.  The
Fubini-Study metric $$\|(\lambda_0 T_0+...+\lambda_n T_n
)(a_0:...:a_n)\|_{FS}=\frac{|\sum \lambda_i a_i |}{\sqrt{\sum_i
a^2_i}}$$ is a smooth metric on $\L_{\C}$. If we take
$\|.\|_1=\|.\|_{FS}$ as our metric at infinity, the limit metric we
obtain is
$$\|(\lambda_0 T_0+...+\lambda_n T_n)(a_0:...:a_n)\|_{nv}=\frac{|\sum \lambda_i
a_i |}{\sup_i(|a_i|)}.$$
\end{example}
\begin{example}
Suppose that $X=E$ is an elliptic curve and assume that $[n] : E
\rightarrow E$ is denoting the multiplication by n on $E$.  As a
consequence of the theorem of the cube, the ample symmetric line
bundle $\L$ on $E$ satisfies $\phi : [n]^*\L \xrightarrow{\sim}
\L^{n^2}$.  The canonical metric is the metric of the cube discussed
in \cite{moretasterisque} and suitable to make $\phi$ and
isomorphism of metrized line bundles.
\end{example}
The following proposition relates the canonical metrics associated
to commuting maps.  It represents the main result of this paper.
\begin{proposition} \label{commuting canonical metrics}
Let $(X,\vfi,\L,\alpha)$ and $(X,\psi,\L,\beta)$ be two polarized
systems on $X$ defined over $K$.  Suppose that the maps $\vfi$ and
$\psi$ satisfy $\vfi \circ \psi = \psi \circ \vfi$, then
$\|.\|_{\vfi}=\|.\|_{\psi}$.
\end{proposition}
\begin{proof}
The key idea is that the canonical metric associate to a morphism
does not depend on the metric we start the iteration with. Let $s
\in \Gamma(X,\L)$ be a non-zero section of $\L$. We are going to
consider two metrics $\|.\|_{v,1}=\|.\|_{\vfi}$ and
$\|.\|'_{v,1}=\|.\|_{\psi}$ on the line bundle $\L$. By our
definition of canonical metric for $\vfi$, we can start with
$\|.\|'_{v,1}$ and obtain $\|s(x)\|_{\vfi}=\lim_k
\|s(\vfi^k(x)\|^{1/\alpha^k}_{\psi}$, but also by our definition of
canonical metric for $\psi$ starting with $\|.\|_{v,1}=\|.\|_{\vfi}$
we get $\|s(x)\|_{\psi}=\lim_l \|s(\vfi^l(x))\|^{1/\beta^l}_{v,1}$.
So using the uniform convergence and the commutativity of the maps,
\begin{equation*}
\begin{split}
\|s(x)\|_{\vfi} & =\lim_{k,l} \|s(\vfi^k \circ \psi^l
(x))\|^{1/\alpha^k \beta^l}_{v,1}  \\ & =\lim_{l,k} \|s(\psi^l \circ
\vfi^k (x))\|^{1/\beta^l \alpha^k}_{v,1}=\|s(x)\|_{\psi},
\end{split}
\end{equation*}
which was the result we wanted to prove.\end{proof}
\subsection{Canonical measures}
Let $X$ be a n-dimensional projective variety defined over a number
field $K$ and suppose that $(X,\vfi,\L,\alpha)$ is a polarized
dynamical system defined over $K$. let $v$ be a place of $K$ over
infinity.  We can consider the morphism $\vfi \otimes v : X_v
\rightarrow X_v$ on the complex variety $X_v$. Associated to $\vfi$
and $v$ we also have the canonical metric $\|.\|_{\vfi,v}$ and
therefore the distribution $c_1(\L,\|.\|_{\vfi,v})=\frac{1}{({\pi}i)
}\partial\overline{\partial}\log \|s_1(P)\|_{\vfi,v}$ analogous to
the first Chern form of $(\L,\|.\|_{\vfi,v})$.  It can be proved
that $c_1(\L,\|.\|_{\vfi,v})$ is a positive current in the sense of
Lelong, and following \cite{dema1} we can define the n-product
$$c_1(\L,\|.\|_{\vfi,v})^n=c_1(\L,\|.\|_{\vfi,v})...c_1(\L,\|.\|_{\vfi,v}),$$
which represents a measure on $X_v$.
\begin{definition} \label{canonical measure}
The measure $d\mu_{\vfi}=c_1(\L_v,\|.\|_{\vfi,v})^n/ \mu(X)$, is
called the canonical measure associated to $\vfi$ and $v$. Once we
have fixed $\L$, it depends only on the metric $\|.\|_{\vfi,v}$.
\end{definition} \label{naive canonical measure}
\begin{example} Consider the rational map $\phi_k : \Pl_{\Q}^n \rightarrow \Pl_{\Q}^n$
given by the formula $\phi_k(T_0:...:T_n)=(T^k_0:...:T^k_n)$, the
canonical measure $d\mu_{\phi_k}$ is the normalize Haar measure on
the n-torus $S^1 \times...\times S^1$.
\end{example}
\begin{example}
Let $E$ be an elliptic curve, $\L$ a symmetric line bundle on $E$
and the map $[n] : E \rightarrow E$.  The canonical measure
associated to $(E,[n],\L,[n]^2)$ can be proved to be
\cite{moretasterisque} the normalized Haar measure on $E$.
\end{example}
\subsection{Canonical heights as intersection numbers}
For a regular projective variety $X$ of dimension n, defined over a
field K, the classical theory of intersection
(\cite{fultoninter},\cite{Serre}) defines the intersection
${c_1}({\mathcal{L}}_{1})...{c_1}({\mathcal{L}}_{n})$ of the classes
${c_1}({\mathcal{L}}_{i})$ associated to line bundles $\L_i$ on $X$,
when $0<i \leq n$.  \\ For the purpose of defining the arithmetic
intersection, we want to assume that $X$ is an arithmetic variety of
dimension $n+1$, that is, given a number field $K$, there exist a
map $f : X \rightarrow \Spec(\cO_K)$, flat and of finite type over
$\Spec(\cO_K)$.  We can define (See for example
\cite{delignedeterminant}, \cite{BGS}, \cite{orsay},
\cite{arakelovinter}, \cite{asterisquedeux} or \cite{zhangvar}) the
arithmetic intersection number
$\tilde{c}_1({\mathcal{L}}_{1})...\tilde{c}_1({\mathcal{L}}_{n+1})$
of the classes $\tilde{c}_1({\mathcal{L}}_{i})$ of hermitian line
bundles $\tilde{\L}_i=(\L_i,\|.\|)$ on $X$.  The fact that
$\tilde{\L}_i$ are hermitian line bundles for $i=1..n+1$, means that
each line bundle $\L_i$ on $X$ is equipped with a hermitian metric
$\|.\|_{v,i}$ over $X_v=X \otimes_K \Spec{\cO_{K_v}}$ for each place
$v$ at infinity.  The numbers
$\tilde{c}_1({\mathcal{L}}_{1})...\tilde{c}_1({\mathcal{L}}_{j})$
prove to be the appropriated theory of intersection in the
particular case of arithmetic varieties, adding places over infinity
allows us to recover the desirable properties of the classical
intersection numbers of varieties over fields.\\
 The last step in the theory of intersection is actually the one that plays
 the more important role in our definition of the canonical height associated to a
 morphism.  Suppose that $X$ is a regular variety of dimension n and $(\L_i,\|.\|_i)_v$ $(i=1,..,p+1)$
 are metrized line bundles on $X$.  Assume also that the $\L_i$ are been equipped with
 semipositive
metrics over all places $v$ (not just at infinity as before) in the
sense of \cite{zhangadelic}.  Such line bundles are called adelic
metrized line bundles and following \cite{zhangadelic}, we can
define the adelic intersection number
$\hat{c}_1({\mathcal{L}}_{1}|Y)...\hat{c}_1({\mathcal{L}}_{p+1}|Y)$
over a $p-$cycle $Y \subset X$.  The adelic intersection number is
in fact a limit of classical numbers
$\tilde{c}_1({\mathcal{L}}_{1})...\tilde{c}_1({\mathcal{L}}_{p+1})$
once the notion of converge is established. The numbers
$\hat{c}_1({\mathcal{L}}_{1}|Y)...\hat{c}_1({\mathcal{L}}_{p+1}|Y)$
satisfy again nice properties, they are multilinear in each of the
$\L_i$ and satisfy
$\hat{c}_1({f^*\mathcal{L}}_{1}|Y)...\hat{c}_1(f^*{\mathcal{L}}_{p+1}|Y)=\hat{c}_1({\mathcal{L}}_{1}|f(Y))...\hat{c}_1({\mathcal{L}}_{p+1}|f(Y)),$
 whenever we have a map $f : X \rightarrow X$.  We are interested in a particular case of this
 situation.  Suppose that we are in the presence of a polarized dynamical system $(X,\vfi,\L,\alpha)$, in this situation
 the canonical metric $\|.\|_{\vfi}$ of \ref{canonical metric} represent a semipositive metric on $\L$, (again we refer to \cite{zhangadelic}) and we can
define the canonical height associated to $(\L,\|.\|_{\vfi})$.
\begin{definition} \label{canonical height as intersection}
 The canonical height $\hat{h}_{\vfi}(Y)$ of a $p-$cycle $Y$ in $X$ is defined as
 $$\hat{h}_{\vfi}(Y)=\frac{\hat{c}_1({\mathcal{L}}|Y)^{p+1}}{(\dim(Y) + 1)c_1(\L|Y)^p}.$$
 It depends only on $(\L,\|.\|_{\vfi})$, where $\|.\|_{\vfi}$ is
 actually representing a collection of canonical metrics over all
 places of $K$.  An important particular case of canonical height
 will be the canonical height $\hat{h}_{\vfi}(P)$ of a point in $P \in X$.
\end{definition}

\begin{example} \label{naive height}Consider the map $\phi_k : \Pl_{\bar{\Q}}^n \rightarrow \Pl_{\bar{\Q}}^n$
given by the formula $\phi_k(T_0:...:T_n)=(T^k_0:...:T^k_n)$, the
canonical height associated to $\phi_k$ is called the naive height
$h_{nv}$ on $\Pl^n$.  If $P=[t_0:...:t_n]$ is a point in $\Pl^n$ the
naive height is
$$h_{nv}([t_0 :...:t_n]) = \frac{1}{[K:\mathbb{Q}]}\log\prod _{\text{places } v
  \text{ of } K} \sup(|t_0|_v ,...,|t_n|_v)^{\N_v},$$
  where $\N_v = [K_v:\bQ_w]$ and $w$ is the place of $\bQ$ such that
$v \mid w$.
\end{example}
\begin{definition} \label{NT}
Let $E$ be an elliptic curve and $\L$ an ample symmetric line bundle
on $E$. The canonical height associated to $[n] : E \rightarrow E$
and $\L$ is called the Neron-Tate height $\hat{h}_{E}$ on $E$.  The
fact that this is independent of n, will be a consequence of
proposition \ref{commuting maps}.
\end{definition}

 The collection of maps $\{\phi_k\}_k$ on $\Pl^n$ and the collection $\{[n]\}_n$ on a given elliptic curve $E$,
 share two important properties, the maps within each collection commute, and share the same canonical height and
 canonical measure.  The following proposition establishes a general fact about canonical
heights and canonical measures of commuting maps on a projective
variety $X$.

\begin{proposition} \label{commuting maps}Let $(X,\vfi,\L,\alpha)$ and $(X,\psi,\L,\beta)$ be two polarized
systems on $X$ defined over $K$.  Suppose that the maps $\vfi$ and
$\psi$ satisfy $\vfi \circ \psi = \psi \circ \vfi$, then
   $ \hat{h}_{\vfi}=\hat{h}_{\psi}=\hat{h}_{\varphi \circ
\psi}$ and
   $ d\mu_{\vfi}=d\mu_{\psi}=d\mu_{\varphi \circ
\psi}.$
\end{proposition}
\begin{proof}
This is a consequence of our definitions of canonical measure
\ref{canonical measure}, canonical height \ref{canonical height as
intersection} and proposition \ref{commuting canonical metrics}.
\end{proof}
\begin{corollary}
  Suppose that two maps $\varphi,
\psi : \Pl^1 \rightarrow \Pl^1$, satisfy the hypothesis of the
previous proposition, then the two maps have the same Julia set.
\end{corollary}
\begin{proof}
The Julia set of a map $\vfi : \Pl^1 \rightarrow \Pl^1$ is nothing
but the closure in $\Pl^1$ of the set of repelling periodic points.
 For details we refer to definition 2.2 in \cite{sz-t-p}.  Now, the corollary is a consequence
of proposition \ref{commuting maps} and proposition 7.2 in
\cite{sz-t-p}.
\end{proof}
\section{Elliptic Curves and examples}
This section illustrates examples of commuting maps on $\Pl^1$. They
all share one thing in common: being induced in some sense by
endomorphisms on elliptic curves.
\begin{proposition} \label{main}
Consider an elliptic curve $E=\C/1 \Z+\tau \Z$ given by Weierstrass
equation $y^2=G(x)$. Suppose that $E$ admits multiplication by the
algebraic number $\lambda$, then multiplication by $\lambda$ in $E$
induces, as quotient by the action of $[-1]$, a map $\vfi_{\lambda}
: \Pl^1 \rightarrow \Pl^1$.  Besides, we have:
\begin{enumerate}
\item
$\hat{h}_{E}(x,y)=\hat{h}_{\lambda}(x)$ for any point $P=(x,y)$ on
$E$.
\item The canonical measure on $\Pl^1$
associated to $\vfi_{\lambda}$ is $$d\mu_{\vfi_\lambda}=\frac{i dz
\wedge d\bar{z}}{2Im(\tau) |G(z)|}.$$
\end{enumerate}
\end{proposition}
\begin{proof}
The first part is a classical fact of the theory of elliptic
functions and complex multiplication. There exist polynomials $P(z)$
and $Q(z)$ where $\deg(P)=\deg(Q)+1=N(\lambda)$ such that $
\wp(\lambda z)=P(\wp(z))/Q(\wp(z))$, where $\wp$ is denoting the
Weierstrass $\wp-$function.  Suppose that we call $\pi$ the quotient
map from $E \rightarrow \Pl^1$, we have a commutative diagram:
\[
\begin{CD}
  E            @>\lambda>>   E   \\
    @V \pi VV            @V \pi VV\\
  \Pl^1   @>\vfi_{\lambda}>> \Pl^1  \\
\end{CD}
\]
Now, consider the line bundle $\L=\cO(1)$ on $\Pl^1$, we have
$\vfi_{\lambda}^* \L \xrightarrow{\sim} \L^{N(\lambda)}$ and equally
for the ample symmetric line bundle $\pi^* \L$ on $E$.  Therefore,
it make sense to talk about canonical heights associated to
$\vfi_{\lambda} : \Pl^1 \rightarrow \Pl^1$ and $\lambda : E
\rightarrow E$.  The number $\lambda$ lies in an imaginary quadratic
extension of $\Q$, so we also have a commutative diagram:
\[
\begin{CD}
  E            @>\lambda>>   E   @>\bar{\lambda}>> E \\
    @V \pi VV            @V \pi VV   @V \pi VV\\
  \Pl^1   @>\vfi_{\lambda}>> \Pl^1   @>\vfi_{\bar{\lambda}}>> \Pl^1 \\
\end{CD}
\]
So, the two maps $\vfi_{\lambda}$ and $\vfi_{\bar{\lambda}}$
commute. After \ref{commuting maps} the canonical height associated
to multiplication by $\lambda$ on $E$ is the same as the canonical
height associated to multiplication by $N(\lambda)$, that is the
Neron-Tate height on $E$.  Take $\L=\cO(1)$ on $\Pl^1$ and $P$ a
point on $E$, the intersection numbers satisfy a projection formula
$$\hat{c}_1({\pi^*\mathcal{L}}|P)=\hat{c}_1({\mathcal{L}}|\pi(P)) \qquad c_1(\pi^*\L |P)=c_1(\L|\pi(P)).$$
 This gives (i) after definition \ref{canonical height as
 intersection}.  For (ii) consider the Haar
 measure $i/2 dz \wedge d\bar{z}$ on $E$, normalized by $Im(\tau)$.
 If $\wp$ denote the Weierstrass function and $\omega=\wp(z)$, we have
 $$\frac{i d\omega \wedge d \bar{\omega}}{2 Im(\tau)}=\frac{i dz \wedge d\bar{z}}{2|\wp'(z)|^2 Im(\tau)}
 =\frac{i dz \wedge d\bar{z}}{2 |y^2| Im(\tau)}=\frac{i dz \wedge d\bar{z}}{2 |G| Im(\tau)}.$$
which gives the result we wanted to prove.
\end{proof}
\begin{remark}
If the elliptic curve $E$ admits multiplication by the numbers
$\lambda$ and $\delta$, then $\vfi_{\lambda} \circ \vfi_{\delta} =
\vfi_{\delta} \circ \vfi_{\lambda}$.
\end{remark}
\begin{example} \label{multipcation by n in E.curves}
Consider an elliptic curve $E$ given by Weierstrass equation $E :
y^2=G(x)$.  For $\lambda=2$ we have
$$\vfi_2(z)=\frac{(G'(z))^2-8zG(z)}{4G(z)}.$$
\end{example}
\begin{example}Let's consider some examples of elliptic curves with complex
multiplication: \end{example}   The elliptic curve $E_1 : y^2=x^3+x$
admits multiplication by $\Z[i]$. The multiplication by $i$ morphism
can be written in $x,y$ coordinates as $[i](x,y)=(-x,iy)$. The two
maps
$$\vfi_{1+i}(z)=\frac{1}{(1+i)^2}\frac{z^2+1}{z}
 \qquad \vfi_{1-i}(z)=-\frac{1}{(1+i)^2}\frac{z^2+1}{z}$$ commute, and their
composition satisfies $$\varphi_{1+i}(\vfi_{1-i}(z))=
\vfi_{1-i}(\varphi_{1+i}(z)) = \varphi_2(z) =
\frac{z^4-2z^2+1}{4(z^3+z)}.$$  The canonical height and measure
are: $$\hat{h}(z)=h_{E_1}(z,\pm \sqrt{z^3+z}) \qquad d\mu(z)=\frac{i
dz \wedge d\bar{z}} {2|z^3+z|}$$ Other examples of maps attached to
$E_1$ are

$$\vfi_{1+2i}(z)=\frac{(-3-4i)z(z^2+1+2i)^2}{(5z^2+1-2i)^2} \quad \vfi_{1-2i}(z)=\frac{(3+4i)z(z^2+1+2i)^2}{(5z^2+1-2i)^2}$$
$$\vfi_{2+i}(z)=\frac{(3-4i)z(z^2+1-2i)^2}{(5z^2+1+2i)^2} \quad
\vfi_{2-i}(z)=\frac{(-3+4i)z(z^2+1-2i)^2}{(5z^2+1+2i)^2}.$$ \\ The
curve $E_2 : y^2=x^3+1$ admits multiplication by the ring $\Z[\rho]$
where $\rho=(\sqrt{-3}+1)/2$. The multiplication by $\rho$ can be
expressed in $x,y$ coordinates as $[\rho](x,y)=(\rho x,y)$.  An
example of commuting maps coming from $E_2$ is
$$\vfi_{\sqrt{-3}}(z)=\frac{-(z^3+4)}{3z^2}  \qquad \vfi_{\sqrt{-3}\rho}(z)=\frac{-\rho(z^3+4)}{3z^2}$$
$$\vfi_{\sqrt{-3}} \circ  \vfi_{\sqrt{-3}\rho} (z) = \vfi_{\varepsilon}(z)=\frac{(z^9-96z^6+48z^3+64)} { 9 \rho z^2(z^3+4)^2},$$

where $\varepsilon=(-3\sqrt{-3}+3)/2$.  The canonical measure
associated to the three maps is
$$d\mu_{E_2}(z)=\frac{\sqrt{3}i dz \wedge d\bar{z}} {3|z^3+1|}.$$

 To have an idea of the
ramification points and indexes of the maps $\vfi_{\lambda}$, we
proof the following lemma:
\begin{lemma}
A ramification point for $\varphi_{\lambda}$ belongs to the image by
$\pi$ of the 2-torsion points on $E$.
\end{lemma}
\begin{proof}
To see this, suppose that $\vfi_{\lambda}^{-1}(\pi(P))=\{ \pi(Q) |
\lambda Q=P \} $ has cardinal strictly smaller than $N(\lambda)$.
Then there exist two points $\pi(Q) \neq \pi(-Q)$ inside the set
$\vfi_{\lambda}^{-1}(P)$, such that $\lambda Q = -\lambda Q$ and
consequently $2 \lambda Q =2P=0.$
\end{proof}

Let's see some examples of the different ramifications that a map
$\vfi_{\lambda}$ may have.  Let $d$ be a positive square free
integer. Assume that the elliptic curve $\C/Z+\sqrt{-d}Z$, admits
multiplication by $\lambda=a+b\sqrt{-d}$.
 Suppose that $P_0=0$, $P_1= 1/2$, $P_2=1/2 + \sqrt{-d}/2$ and
$P_3=\sqrt{-d}/2$ denote the 2-torsion points on $E$ and that $r_j$
denotes the amount of pre-images of the point $\pi(P_j)$, that is,
the cardinality of the set $\vfi_{\lambda}^{-1}(\pi(P_j))$. Under
the conditions previously described, we can observe for example that
for $\lambda=2$, the points in $\vfi_{2}^{-1} (\pi(P_0))$ are not
ramification points of $\vfi_2$. On the other hand for the
multiplication by $\lambda=1+2i$ on $E_1$, all points in
$\vfi_{1+2i}^{-1} (\pi(P_0))\cup \vfi_{1+2i}^{-1} (\pi(P_1)) \cup
\vfi_{1+2i}^{-1} (\pi(P_2)) \cup \vfi_{1+2i}^{-1} (\pi(P_3))$ are
ramification points of $\vfi_{1+2i}$.  The following table summarize
the results:
\begin{center}
\begin{tabular}{|c|c|c|} \hline  $\lambda=a+b\sqrt{-d}$, $N(\lambda)$
& $r_j, j=0,2$ & $r_j, j=1,3$
\\[4pt] \hline
 $a+bd \equiv 1 mod(2)$ & $ r_0=(N(\lambda)+1)/2$& $r_1=(N(\lambda)+1)/2$
 \\
\cline{2-3}
 &$ r_2=(N(\lambda)+1)/2$& $ r_3=(N(\lambda)+1)/2$\\ [4pt]
 \hline
 $a \equiv b \equiv 0 mod(2)$, $N(\lambda)>4$ & $ r_0=N(\lambda)/2+2$& $ r_1=N(\lambda)/2$
 \\ [4pt]
\cline{2-3}
 &$r_2=N(\lambda)/2$& $r_3=N(\lambda)/2$\\ [4pt]
\hline
 $a \equiv b \equiv 0 mod(2)$, $N(\lambda)=4$ & $ r_0=4$& $r_1=N(\lambda)/2$
 \\ [4pt]
\cline{2-3}
 &$ r_2=N(\lambda)/2$& $ r_3=N(\lambda)/2$\\ [4pt]
 \hline
 $a \equiv bd \equiv 1 mod(2)$, $N(\lambda)=2$ & $ r_0=1$& $ r_1=N(\lambda)$
 \\ [4pt]
\cline{2-3}
 &$ r_2=1$& $ r_3=N(\lambda)$\\ [4pt]
 \hline
 $a \equiv bd \equiv 1 mod(2)$, $N(\lambda)>2$ & $ r_0=N(\lambda)/2$& $ r_1=N(\lambda)/2+1$
 \\ [4pt]
\cline{2-3}
 &$ r_2=N(\lambda)/2$& $ r_3=N(\lambda)/2+1$\\ [4pt]
\hline

 \hline
\end{tabular}
\end{center}

\end{document}